\documentclass{amsart}

\usepackage{amsthm,amsmath,amsthm,amssymb,mathrsfs,graphicx,mathtools,relsize,hyperref,cleveref,enumerate,nameref,dsfont}
\usepackage[utf8]{inputenc}
\usepackage[all]{xy}
\usepackage[nobysame]{amsrefs}

\theoremstyle{theorem}
 \newtheorem{thm}{Theorem}[section]
 \newtheorem{prop}[thm]{Proposition}
 \newtheorem{lem}[thm]{Lemma}
 \newtheorem{cor}[thm]{Corollary}

 \newtheorem*{thmA}{Theorem A}
 \newtheorem*{thmB}{Theorem B}
\theoremstyle{definition}
 
 \newtheorem{dfn}[thm]{Definition}
 
\theoremstyle{remark}
 \newtheorem{rem}[thm]{Remark}
 \numberwithin{equation}{section}
 
\renewcommand{\le}{\leqslant}\renewcommand{\leq}{\leqslant}
\renewcommand{\ge}{\geqslant}
\renewcommand{\setminus}{\smallsetminus}
\def\Aut{\text{\rm Aut}}

\def\id{\text{\rm id}}

\def\COS{\text{\rm COS}}

\def\rank{\text{\rm rank}}

\def\ker{\text{\rm ker}}
\def\Re{\text{\rm Re}}

\def\H{\mathcal{H}}

\def\A{\mathcal{A}}

\def\B{\mathcal{B}}

\def\ind{\text{\rm ind}}
\def\cind{\text{\rm c-ind}}

\newcommand{\norm}[1]{\left\lVert#1\right\rVert}

\newcommand{\hlight}[1]{\textit{\textbf{#1}}}

\newcommand{\dd}{{\rm \, d}}
\newcommand{\spn}{{\rm span \,}}

\title[Non-Wiener groups with a Gelfand pair]{Non-Wiener groups with a Gelfand pair}
\date{\today}

\thanks{M.C.\ acknowledges support from the FWO and F.R.S.-FNRS under the Excellence of Science program (project ID 40007542).}

\author[Max Carter]{Max Carter} 
\address{Institut de recherche en mathématique et physique \\
Université Catholique de Louvain \\ 
Chemin du Cyclotron 2 \\
boîte L7.01.02 \\
1348 Louvain-la-Neuve \\
Belgique.}
\email{max.carter@uclouvain.be}

\author[Jared T.\ White]{Jared T.\ White} 
\address{School of Mathematics and Statistics \\
 The Open University \\
 Walton Hall, Milton Keynes MK7 6AA \\
 United Kingdom.}
\email{jared.white@open.ac.uk}

\begin{document}

\begin{abstract}
Let $G$ be a non-amenable locally compact group and $K$ a compact subgroup of $G$ such that $(G,K)$ is a Gelfand pair. We show that if $G$ admits a suitable boundary representation which is topologically irreducible and not unitarizable, then $G$ is not a 
Wiener group in the sense that its Fourier transform does not satisfy the analogue of Wiener's Tauberian theorem. As an application, we show that if $G$ is a closed non-compact boundary transitive group of automorphisms of a connected locally finite graph with infinitely many ends, or a non-abelian split reductive algebraic group over a non-archimedean local field, then $G$ is not Wiener. 

\end{abstract}

\maketitle


\section{Introduction}

In 1932, Norbert Wiener's article entitled \textit{Tauberian Theorems} was published in the Annals of Mathematics \cite{Wie32}. This article contains a number of different well known theorems in Fourier analysis, many of which receive the label of ``Wiener's Tauberian theorem'' in the literature. One of these results can be formulated as follows: a function $f \in L^1(\mathbb{R}^d)$ $(d \ge 1)$ generates a dense ideal of $L^1(\mathbb{R}^d)$ if and only if the Fourier transform of $f$ vanishes nowhere. In the terminology of the present article, this means that the group $\mathbb{R}^d$ is a Wiener group.

We shall now discuss how one can formulate a version of this property for general locally compact groups. Let $G$ be a (possibly non-abelian) locally compact group and let $\widehat{G}$ denote the set of unitary equivalence classes of irreducible unitary representations on $G$. Given $f \in L^1(G)$ and a unitary representation $(\pi,\H_\pi)$ of $G$, one may define the Fourier transform of $f$ evaluated at $\pi$ as the operator valued integral 
\begin{displaymath} \hat{f}(\pi) := \int_G f(x) \pi(x) \dd x. \end{displaymath}  
The integral here is defined in the sense of Bochner and $\hat{f}(\pi)$ is an element of $\B(\H_\pi)$, the space of bounded operators on the Hilbert space $\H_\pi$. Although the Fourier transform of $f$ is defined on all unitary representations of $G$, we will from here on in consider $\hat{f}$ as a function on the unitary dual $\widehat{G}$ of $G$, in analogue with the case of abelian groups. Furthermore, when $G$ is an abelian group, the irreducible unitary representations of $G$ can be identified with the unitary characters of $G$, so the above integral becomes an integral of complex valued functions, and we have the usual Fourier transform on an abelian group.

We remind the reader that there is a bijective correspondence between the (equivalence classes of) irreducible unitary representations of $G$ and the topologically irreducible $*$-representations of the Banach $*$-algebra $L^1(G)$ \cite[$\S$3.2]{Fol16}. This correspondence is given by the map
\begin{displaymath}  \widehat{G} \rightarrow \widehat{L^1(G)}, \pi \mapsto \widetilde\pi  \end{displaymath}
where $\widetilde\pi(f) := \hat{f}(\pi)$ for each $f \in L^1(G)$. 

From the above facts, one can then deduce that the statement that the Fourier transform of $f \in L^1(\mathbb{R}^d)$ vanishes nowhere from the first paragraph is equivalent to requiring that $f$, and hence the ideal generated by $f$, is not contained in the kernel of any 1-dimensional $*$-representation of $L^1(G)$. Motivated by these facts, one defines the notion of a Wiener group as follows.

\begin{dfn}
A locally compact group $G$ is called \hlight{Wiener} if every proper closed two-sided ideal $I \subset L^1(G)$ is contained in the kernel of a topologically irreducible $*$-representation of $L^1(G)$. 
\end{dfn}

We remark that the study of Wiener's Tauberian theorem and Wiener groups is intimately related to the topic of spectral synthesis \cite{Ben75}. Indeed, given a closed two-sided ideal $I \subseteq L^1(G)$, one defines the following subset of $\widehat{G}$, called the \hlight{hull} of $I$:
\begin{displaymath} h(I) := \{ \pi \in \widehat{G} : \hat{f}(\pi) = 0 \: \forall f \in I \}.\end{displaymath}
Similarly, given $S \subseteq \widehat{G}$ closed, one defines a closed two-sided ideal in $L^1(G)$, called the \hlight{kernel} of $S$:
\begin{displaymath} k(S) := \{ f \in L^1(G) : \hat{f}(\pi) = 0 \: \forall \pi \in S \}.  \end{displaymath}
It can be shown that $k(S)$ is the largest closed two-sided ideal $I \subseteq L^1(G)$ satisfying $h(I) = S$. The set $S \subseteq G$ is called a \hlight{set of synthesis} if $k(S)$ is the only closed ideal in $L^1(G)$ with hull equal to $S$.  The topic of spectral synthesis is an extremely interesting and heavily studied topic, particularly in the cases of abelian groups (see \cite[Chapter 10]{HR70}, \cite[Chapter 7]{Rud62} and \cite[Chapter 6]{KL18}) and connected nilpotent Lie groups \cite{Lud83a,LMB10,LMBP13,BL16}. The property of a group $G$ being Wiener is equivalent to the empty set $\emptyset \subset \widehat{G}$ being a set of synthesis.

The question of which groups are Wiener has been investigated extensively, particularly during the mid-to-late 20th century, and the focus during this period was primarily on understanding this property for connected locally compact groups and discrete groups \cite{Lep84}. There still remains, however, a strong lack of understanding of which non-discrete totally disconnected locally compact groups (abbreviated \hlight{tdlc groups} from now on) are Wiener. A main point of this paper is to make further progress on the Wiener property for this class of groups. We note that this article fits into a broader project of the first author concerning progressing the harmonic analysis and representation theory of non-discrete tdlc groups \cite{Car24,Car25-1,Car26,CC25}.

As already mentioned, Wiener's result implies that the group $\mathbb{R}^d$ is Wiener for any $d \ge 1$. It is also a well known and classical result in harmonic analysis that every locally compact abelian group is Wiener. We are not sure who this result is originally due to, but one may like to consult \cite[$\S$7.2]{Rud62} and the references therein for more information. Other classes of groups that are known to be Wiener include nilpotent groups \cite{Lud79} and compactly generated groups with polynomial growth \cite{Los01}. In the context of solvable groups, there is exactly one connected exponential (hence solvable) Lie group with real dimension $\le 4$ that is not Wiener \cite{LP79}.

On the other hand, as far as the authors are aware, there are no known examples of non-amenable Wiener groups. It is also well known that every non-compact connected semisimple Lie group is not Wiener \cite[Appendix]{Lep76}. Other than some select examples of solvable Lie groups, these are the only known examples of non-Wiener groups. The point of this article is to develop a general method for showing that a non-amenable group with a Gelfand pair is not Wiener. We use our method to prove that many non-amenable tdlc groups are not Wiener, significantly expanding the class of non-Wiener groups.

The setup in this article is as follows. Let $G$ be a locally compact group and suppose that $G$ contains a compact subgroup $K$ such that the convolution algebra of integrable $K$-bi-invariant functions on $G$, denoted by $L^1(K \backslash G /K)$, is commutative. When such a subgroup $K$ exists, the pair $(G,K)$ is called a \hlight{Gelfand pair}. It is shown in \cite{Mon20} that there exists a maximal cocompact amenable subgroup $P$ of $G$ such that $G=KP$ and such that we have homeomorphisms
\begin{displaymath} \partial G \cong G/P \cong K/K\cap P \end{displaymath}
where $\partial G$ denotes the Furstenberg boundary of $G$. 

Now suppose that $G$ is non-amenable. Then, $G/P$ is non-trivial and compact. The measure on $G/P$ obtained by pushing forward the Haar measure of $K$ via the homeomorphism $G/P \cong K/K\cap P$ is $K$-invariant and its measure class is preserved by the $G$-action. Denote this measure by $\mu$. For any complex parameter $z \in \mathbb{C}$, define a representation $\pi_{z}: G \rightarrow \B(L^2(G/P))$ given by 
\begin{displaymath} \pi_{z}(g)f(xP) := \bigg(\frac{\dd g\mu}{\dd \mu}(xP)\bigg)^zf(g^{-1}xP)  \end{displaymath}
where $g \in G$ and $f \in L^2(G/P)$. This representation is always strongly continuous.

The main result of this paper is the following theorem. We note that $\mathds{1}_{G/P}$ denotes the function which is identically 1 on $G/P$.
\begin{thmA}
Let $G$ be a non-amenable locally compact group and $K$ a compact subgroup of $G$ such that $(G,K)$ is a Gelfand pair. Suppose that there exists $z \in \mathbb{C}$ with $0 < \Re(z) < 1/2$ such that the representation $\pi_{z}$ is topologically irreducible and the matrix coefficient
\begin{displaymath} \varphi_z(g):= \langle \pi_{z}(g) \mathds{1}_{G/P}, \mathds{1}_{G/P} \rangle_{L^2(G/P)} \end{displaymath}
is not positive-definite. Then $G$ is not Wiener. 
\end{thmA}

We remark that, by Proposition \ref{prop:unitarizable}, the matrix coefficient $\varphi_z$ is not positive-definite in the case that the representation $\pi_z$ is not unitarizable. This fact is very useful in applications of the theorem.

We apply Theorem A to give explicit examples of non-Wiener tdlc groups. To do this, we need to study principal series representations of the following groups deeply, which is completed in the proof of Theorem B.

\begin{thmB}
\begin{enumerate}[(i)]
   \item Let $X$ be a connected locally finite graph with infinitely many ends. If $G \le \Aut(X)$ is closed, non-compact and acts transitively on the set of ends of $X$, then $G$ is not Wiener.
   \item Any non-abelian split reductive algebraic group over a non-archimedean local field is not Wiener. 
\end{enumerate}
\end{thmB}

Part $(i)$ of Theorem B resolves an open problem from T.\ Palmer's well known two volume encyclopaedia on Banach $*$-algebras \cite{Pal01} regarding whether the group $\Aut(T)$ is Wiener, where $T$ is a regular tree (see the last line of the table on page 1490).

\section{Preliminaries on representations and Gelfand pairs}

Here we shall layout some basic definitions and results that will be used throughout the article.  We assume a rudimentary knowledge of the theory of topological groups, Banach algebras and their various representation theories. We refer the reader to \cite{vD09} or \cite{Sil79} for any unproven facts about analysis on Gelfand pairs and representations respectively.

\subsection{Conventions}
Throughout the article, an ideal is understood to be two-sided unless otherwise stated. Integration on a locally compact group $G$ is always performed with respect to some prior fixed left Haar measure.

\subsection{Gelfand pairs}

We now introduce some of the basic theory of Gelfand pairs which is the centre point of our article.

Let $G$ be a unimodular locally compact group and $K$ a compact subgroup of $G$. Recall that the space $C_c(G)$ of continuous compactly supported functions on $G$ is a $*$-algebra when equipped with the convolution product and involution $f^*(x) := \overline{f(x^{-1})}$.

Corresponding to the pair $(G,K)$ is the so called \hlight{Hecke algebra} defined as
\begin{displaymath} C_c(K\backslash G/K):= \{ f \in C_c(G) : f(kgk') = f(g) \: \forall k,k'\in K,\forall g \in G\}. \end{displaymath}
This is a $*$-subalgebra of $C_c(G)$ consisting of the functions which are \hlight{$K$-bi-invariant}. Of course, we have natural inclusions $C_c(K\backslash G/K) \subseteq C_c(G) \subseteq L^1(G)$, hence we may also complete the $*$-algebra $C_c(K\backslash G/K)$ in the $L^1$-norm. This completion, denote by $L^1(K \backslash G / K)$, is a Banach $*$-subalgebra of $L^1(G)$. 

\begin{dfn}
Let $G$ be a locally compact group and $K$ a compact subgroup of $G$. The pair $(G,K)$ is called a \hlight{Gelfand pair} if the algebra $C_c(K \backslash G /K)$ is commutative. 
\end{dfn}

We remark that, if $(G,K)$ is a Gelfand pair, then, since $C_c(K\backslash G/K)$ is dense in $L^1(K \backslash G /K)$, it follows that $L^1(K \backslash G /K)$ is also commutative. Thus $L^1(K \backslash G /K)$ is a commutative Banach $*$-algebra. We shall now say a bit about the dual of $L^1(K \backslash G/K)$. To do this, we must define the notion of a spherical function corresponding to the pair $(G,K)$.

\begin{dfn}
Let $G$ be a locally compact group and $K$ a compact subgroup of $G$ such that $(G,K)$ is a Gelfand pair. A \hlight{spherical function} on the pair $(G,K)$ (or for short, a spherical function on $G$) is a continuous $K$-bi-invariant function $\varphi:G \rightarrow \mathbb{C}$ such that 
\begin{displaymath} \chi_\varphi(f) := \int_G f(x) \varphi(x^{-1}) \dd x  \end{displaymath}
is a non-trivial multiplicative linear functional of the algebra $C_c(K \backslash G /K)$. 
\end{dfn}

The following result gives a number of equivalent definitions of a spherical function.

\begin{prop}\cite[$\S$6.1(ii)]{vD09}
Let $G$ be a locally compact group and $K$ a compact subgroup of $G$ such that $(G,K)$ is a Gelfand pair. Let $\varphi: G \rightarrow \mathbb{C}$ be a continuous $K$-bi-invariant function such that $\varphi(\id_G) = 1$. Then the following are equivalent:
\begin{enumerate}[(i)]
   \item $\varphi$ is a spherical function;
   \item For all $x,y \in G$ we have that 
   \begin{displaymath} \int_K \varphi(xky) \dd k = \varphi(x)\varphi(y) \end{displaymath}
   where $\dd k$ denotes the normalised Haar measure on $K$;
   \item For every $f \in C_c(K \backslash G/K)$, there exists a complex number $\lambda(f,\varphi)$ such that $f*\varphi = \lambda(f,\varphi)\varphi$.
\end{enumerate}
\end{prop}

For many examples of Gelfand pairs, the spherical functions can be explicitly determined in a natural way; see for example \cite[$\S$7]{vD09}.

Now let $\varphi$ be a spherical function on $G$. Then, as mentioned above, the function
\begin{displaymath} \chi_\varphi(f) = \int_G f(x)\varphi(x^{-1}) \dd x \end{displaymath}
is a multiplicative linear functional on $C_c(K \backslash G /K)$. If $\varphi$ is furthermore assumed to be bounded, then by density of $C_c(K \backslash G /K)$ in $L^1(K \backslash G/K)$, the functional $\chi_\varphi$ extends to a multiplicative linear functional of $L^1(K \backslash G/K)$ which we will also denote by $\chi_\varphi$. It is also true that all multiplicative linear functionals of $L^1(K \backslash G/K)$ are formed in this way.

\begin{prop}\cite[Theorem 6.1.7]{vD09}
Let $G$ be a locally compact group and $K$ a compact subgroup of $G$ such that $(G,K)$ is a Gelfand pair. For every bounded spherical function $\varphi$ on $G$, the function
\begin{displaymath} \chi_\varphi(f) = \int_G f(x)\varphi(x^{-1}) \dd x \end{displaymath}
is a multiplicative linear functional on $L^1(K \backslash G /K)$. Conversely, every multiplicative linear functional on $L^1(K \backslash G /K)$ is of this form. 
\end{prop}

So we now have a description of the multiplicative linear functionals on $C_c(K \backslash G /K)$ and $L^1(K \backslash G /K)$. We shall now give a description of the $*$-homomorphisms $\chi: L^1(K \backslash G /K) \rightarrow \mathbb{C}$. In particular, we need to determine which multiplicative linear functionals $\chi$ of $L^1(K \backslash G /K)$ satisfy $\chi(f^*) = \overline{\chi(f)}$ for all $f \in L^1(K \backslash G /K)$. To do this, we need to consider the positive definite spherical functions.

\begin{dfn}
Let $G$ be a locally compact group and $\varphi: G \rightarrow \mathbb{C}$ a continous function on $G$. Then $\varphi$ is called \hlight{positive definite} if for all $f \in C_c(G)$
\begin{displaymath} \int_G\int_G f(x)\overline{f(y)} \varphi(y^{-1}x) \dd x \dd y \ge 0. \end{displaymath}
\end{dfn}

We then have the following result. 

\begin{prop}\cite[Lemma 5.1.8]{vD09}\label{prop:posdef}
Let $G$ be a locally compact group and $K$ a compact subgroup of $G$ such that $(G,K)$ is a Gelfand pair. Let $\varphi$ be a spherical function on $(G,K)$. If $\varphi$ is positive definite, then for all $g \in G$, $\varphi(g) = \overline{\varphi(g^{-1})} =: \varphi^*(g)$. In particular, $\chi_\varphi(f^*) = \overline{\chi_\varphi(f)}$ for all $f \in C_c(G)$, so $\chi_\varphi: L^1(K \backslash G / K) \rightarrow \mathbb{C}$ is a $*$-homomorphism if $\varphi$ is positive-definite.
\end{prop}

In a similar light, we also have the following equivalences.

\begin{prop}\cite[Lemma 9.2.5]{Wol07}\label{prop:posdefsph}
Let $G$ be a locally compact group and $K$ a compact subgroup of $G$ such that $(G,K)$ is a Gelfand pair. Let $\varphi$ be a spherical function on $(G,K)$. Then the following are equivalent:
\begin{enumerate}[(i)]
   \item $\varphi$ is positive definite;\
   \item $\chi_\varphi(f*f^*) \ge 0$ for all $f \in C_c(K \backslash G /K)$.
\end{enumerate}
\noindent If $\varphi$ is furthermore assumed to be bounded, then $(i)$ and $(ii)$ are equivalent to:
\begin{enumerate}[(i)]
   \item[(iii)] $\chi_\varphi(f*f^*) \ge 0$ for all $f \in L^1(K \backslash G /K)$.
\end{enumerate}
\end{prop}

A consequence of Proposition \ref{prop:posdef} is the following.

\begin{cor}\label{cor:starclosed}
Let $G$ be a locally compact group and $K$ a compact subgroup of $G$ such that $(G,K)$ is a Gelfand pair. Suppose that there exists a bounded spherical function $\varphi$ on $G$ that is not positive definite. Then, the kernel of $\chi_\varphi$ in $L^1(K \backslash G/K)$ is not $*$-closed.
\end{cor}

\begin{proof}
Suppose for a contradiction that $\ker_{L^1(K\backslash G/K)}(\chi_\varphi)$ is $*$-closed. Let $f \in \ker_{L^1(K\backslash G/K)}(\chi_\varphi)$. We compute that 
    \begin{align*}
        \chi_\varphi(f^*) &= \int_G f^*(x) \varphi(x^{-1}) \dd x = \int_G \overline{f(x^{-1})}\varphi(x^{-1}) \dd x \\
        &= \overline{\int_G f(x)\overline{\varphi(x)} \dd x} = \overline{\int_G f(x) \varphi^*(x^{-1}) \dd x} = \overline{\chi_{\varphi^*}(f)}.
    \end{align*}
In particular, $f^* \in \ker_{L^1(K\backslash G/K)}(\chi_\varphi)$ if and only if $f \in \ker_{L^1(K\backslash G/K)}(\chi_{\varphi^*})$. Since $\ker_{L^1(K\backslash G/K)}(\chi_\varphi)$ is $*$-closed by assumption, one deduces that $\ker_{L^1(K\backslash G/K)}(\chi_\varphi) = \ker_{L^1(K\backslash G/K)}(\chi_{\varphi^*})$. This implies that the span of $\chi_\varphi$ and the span of $\chi_{\varphi^*}$ are equal in the dual $L^1(K\backslash G/K)^*$. In particular, there exists a constant $\lambda \in \mathbb{C}$ such that $\lambda\chi_{\varphi} = \chi_{\varphi^*}$. Now let $\mathds{1}_K \in L^1(K\backslash G /K)$ be the characteristic function on $K$. Assume that the Haar measure on $K$ has been normalised to have volume 1. Then one computes that $\chi_{\varphi^*}(\mathds{1}_K) = 1$ and $\lambda\chi_{\varphi}(\mathds{1}_K) = \lambda$. Thus $\lambda =1$ since $\lambda \chi_\varphi = \chi_{\varphi^*}$, and so $\chi_{\varphi} = \chi_{\varphi^*}$. It follows from this that $\varphi = \varphi^*$. But, by Proposition \ref{prop:posdef}, this contradicts the fact that $\varphi$ is assumed to not be positive definite.
\end{proof}

\subsection{Smooth and admissible representations}

We start by recalling some definitions. Throughout this section all vector spaces will be over the field of complex numbers. A representation of a locally compact group $G$ is a pair $(\pi,V)$, where $V$ is a vector space and $\pi: G \rightarrow GL(V)$ is a homomorphism.

\begin{dfn}
Let $G$ be a tdlc group and $(\pi,V)$ a representation of $G$.
\begin{enumerate}[(i)]
   \item A vector $v \in V$ is called \hlight{smooth} if the subgroup $G_v := \{ g \in G : \pi(g)v = v\}$ of $G$ is open. \
   \item The representation $\pi$ is called \hlight{smooth} if every $v \in V$ is smooth. \
   \item The representation $\pi$ is \hlight{admissible} if for every compact open subgroup $K \le G$, the subspace of fixed vectors $V^K:= \{ v \in V : \pi(k)v = v \; \forall k \in K\}$ is finite-dimensional. \
   \item Suppose that there is a norm $\norm{\cdot}$ on $V$. Then $\pi$ is \hlight{continuous} if for all $v \in V$, the map $g \mapsto \pi(g)v$ is continuous with respect to the group topology on $G$ and the topology induced by the norm on $V$. \
\end{enumerate}
\end{dfn}

There exists the following equivalent characterisation of a smooth representation. We denote \hlight{the set of compact open subgroups} of a tdlc group by $\COS(G)$. The proof of the following proposition is obvious after noting that every open subgroup of a tdlc group contains a compact open subgroup.

\begin{prop}\label{prop:smoothdfn}
Let $G$ be a tdlc group and $(\pi, V)$ a representation of $G$. Then the following are equivalent:
\begin{enumerate}
   \item[(i)] $\pi$ is smooth;
   \item[(ii)] $V = \bigcup_{K \in \COS(G)} V^K$.
\end{enumerate}
\end{prop}

We also note the following well known equivalences of the property of a representation being admissible. See for example, \cite[$\S$1.5]{Sil79}.

\begin{prop}\label{prop:admissibledfn}
Let $G$ be a tdlc group and $(\pi,V)$ a smooth representation of $G$. Then the following are equivalent:
\begin{enumerate}[(i)]
   \item $(\pi,V)$ is admissible;
   \item For every $K \in \COS(G)$, $V$ decomposes as a direct sum of irreducible smooth $K$-modules each occuring with finite multiplicity.
\end{enumerate}
\end{prop}

Given any representation $(\pi,V)$ of a tdlc group $G$, the space $V^\infty := \bigcup_{K \in \COS(G)} V^K$ is a subspace of $V$ invariant under the action of $\pi$, and it is precisely the subspace of smooth vectors in $V$.

\begin{prop}\label{prop:completionrep}
Let $G$ be a tdlc group and $(\pi, V)$ a smooth admissible representation of $G$. Let $\norm{\cdot}$ be a norm on $V$ with respect to which $\pi$ is continuous and denote by $(\widetilde\pi,\widetilde V)$ the completion of this representation with respect to this norm. Then the following are true:
\begin{enumerate}[(i)]
   \item $\widetilde{V}^K = V^K$ for all compact open subgroups $K$ of $V$;
   \item $\widetilde{V}^\infty = V$;
   \item Every closed $G$-invariant subspace of $\widetilde{V}$ intersects $V$ non-trivially. In particular, $(\pi,V)$ is algebraically irreducible if and only if $(\widetilde\pi,\widetilde V)$ is topologically irreducible. 
\end{enumerate}
\end{prop}

\begin{proof}
$(i)$ Since $\pi$ is a smooth admissible representation, $V^K$ is finite-dimensional for all compact open subgroups $K \in \COS(G)$. Now fix a compact open subgroup $K$ of $G$ and suppose that $\widetilde{V}^K \ne V^K$. Let $v \in \widetilde{V}^K \setminus V^K$. By density of $V$ in $\widetilde{V}$, there exists a sequence $(v_n)_{n=1}^\infty \subset V$ such that $v_n \rightarrow v$ as $n \rightarrow \infty$. Then, since the projection 
\begin{displaymath} P^K:\widetilde{V} \rightarrow \widetilde{V}^K, v \mapsto \int_K \widetilde\pi(k)v \dd k \end{displaymath}
is continuous, it follows that $P^K(v_n) \rightarrow P^K(v) = v$ as $n \rightarrow \infty$. But since $v_n \in V$ for each $n$, $P^K(v_n) \in V^K$ for each $n$. So we have found a sequence in $V^K$ converging to $v$ in $\widetilde{V}$. But $V^K$ being finite-dimensional must be closed in $\widetilde{V}$, so $v \in V^K$. 

$(ii)$ By part $(i)$, $\widetilde{V}^K = V^K$ for each $K \in \COS(G)$. Thus $\widetilde{V}^\infty  = \bigcup_{K \in \COS(G)} \widetilde{V}^K = \bigcup_{K \in \COS(G)} V^K = V$.

$(iii)$ Suppose that $W \subseteq \widetilde V$ is a closed $G$-invariant subspace and let $w \in W$ be non-zero. To show that $W$ intersects $V$ non-trivially, it suffices to show that there exists a compact open subgroup $K \in \COS(G)$ such that $P^K(w) \ne 0$. Suppose for a contradiction that this is not the case, that is, for every $K \in \COS(G)$, $P^{K}(w) = 0$.

Let $\mu$ denote the Haar measure on $G$. Let $X:=\{ K \in \COS(G) : \mu(K) \le 1 \}$. Note that $X$ is a directed set when equipped with reverse set inclusion as the ordering. Then, $(P^K(w))_{K \in X}$ is a net in $\widetilde{V}$. We claim that it converges to $w$. Indeed, for any $K \in X$
\begin{align*} 
\norm{P^K(w) - w} &= \norm{\int_K \widetilde{\pi}(k)w \dd k - w}  \\
&\le \sup_{k \in U} \norm{\widetilde{\pi}(k)w - w}.
\end{align*} 

So by continuity of $\widetilde{\pi}$ (at the identity of $G$), for every $\epsilon > 0$, there exists a compact open subgroup $K_0 \in X$ such that 
\begin{displaymath} \norm{P^{K_0}(w)-w} \le \sup_{k \in K_0} \norm{\widetilde{\pi}(k)w -w} < \epsilon. \end{displaymath}
This implies that the net $(P^K(w))_{K \in X}$ converges to $w$. Since $w$ is assumed to be non-zero, there must exist a compact open subgroup $K \in X$ such that $P^K(w) \ne 0$. This implies that $P^K(\widetilde{V}) \subseteq V^K$ intersects $W$ non-trivially. 

To prove the second claim of $(iii)$, suppose that $(\pi,V)$ is algebraically irreducible, and let $W$ be a closed invariant subspace of $\widetilde{V}$. Then by the previous argument, $V \cap W$ is a non-trivial invariant subspace of $V$, so we must have that $V \subseteq W$ since $V$ is algebraically irreducible. This implies that $W=\widetilde{V}$ since $W$ is closed and $V$ is dense in $\widetilde{V}$. So $(\widetilde{\pi},\widetilde{V})$ is topologically irreducible. 

To prove the converse, suppose that $(\widetilde{\pi},\widetilde{V})$ is topologically irreducible. Then, let $W$ be a non-trivial $G$-invariant subspace of $V$. Note that $W$ is smooth and admissible since it is a subrepresentation of a representation satisfying these properties. We need to show that $W=V$. Since the closure of $W$ in $\widetilde{V}$ is a closed non-trivial $G$-invariant subspace of $\widetilde{V}$, it follows that $W$ must be dense in $\widetilde{V}$, since $\widetilde{V}$ is topologically irreducible. Then, for every compact open subgroup $K$ of $G$, the projection $P^K$ is continuous. Thus it follows that $P^K(W) = W^K$ is dense in $P^K(\widetilde{V}) = \widetilde{V}^K = V^K$ for every $K \in \COS(G)$. But $W^K$ is finite-dimensional since $W$ is admissible, so $W^K$ must be closed, hence $W^K=V^K$ for every $K \in \COS(G)$. But since both $W$ and $V$ are smooth, we have that $W = \bigcup_{K \in \COS(G)} W^K = \bigcup_{K \in \COS(G)} V^K = V.$ This implies that $V$ is algebraically irreducible and completes the proof.
\end{proof}

Finally, we have the following result which will be used later in the article.

\begin{prop}\label{prop:unitarizable}
Let $G$ be a tdlc group and $(\pi,V)$ a smooth admissible algebraically irreducible representation of $G$. Let $\langle \cdot, \cdot \rangle: V \times V \rightarrow \mathbb{C}$ be an inner-product on $V$ with respect to which $\pi$ is continuous but not necessarily unitary. Suppose that there exists a $v \in V$ such that the matrix coefficient $\varphi(g) := \langle \pi(g)v,v \rangle$ is positive-definite. Then, the representation $(\pi,V)$ must be unitarizable.
\end{prop}

\begin{proof}
Since $\text{span}\{\pi(g)v : g \in G \}$ is a $G$-invariant subspace of $V$, and since $(\pi,V)$ is algebraically irreducible, we must have that $V = \text{span}\{\pi(g)v : g \in G \}$. In particular, every element of $V$ can be written as a sum of the form $\sum_{i=1}^n c_i \pi(g_i) v$ with $c_i \in \mathbb{C}$ and $g_i \in G$ for each $i$.

Now define a bilinear form $B:V\times V \rightarrow \mathbb{C}$ by
\begin{displaymath} B\bigg(\sum_{i=1}^n c_i \pi(g_i) v, \sum_{j=1}^m d_j \pi(h_j) v \bigg) := \sum_{i,j} c_i \overline{d_j} \varphi(h_j^{-1}g_i).\end{displaymath} 
One checks easily that $B$ is $G$-invariant. It follows from Proposition \ref{prop:posdefsph} that $B$ is Hermitian and $B$ is positive since $\varphi$ is positive-definite. We now claim that $B$ is definite. Indeed, $\ker(B) := \{ w \in V : B(w,w) = 0\}$ is a proper $G$-invariant subspace of $V$ by \cite[Theorem A.3]{Fol16}, so $\ker(B)$ must be trivial since $(\pi,V)$ is algebraically irreducible. Thus $B$ is a $G$-invariant positive-definite Hermitian form on $V$, in particular a $G$-invariant inner-product, so $(\pi,V)$ is unitarizable.
\end{proof}

\subsection{Contragradient representations and matrix coefficients}

Here we introduce some notation concerning contragradient representations and matrix coefficients.

Let $G$ be a tdlc group and $(\pi,V)$ a smooth representation of $G$. Throughout this section and the remainder of the article, $\widehat{V}$ will be used to denote the (complex) vector space dual of $V$. There is a representation $\widehat{\pi}:G \rightarrow GL(\widehat{V})$ given by 
\begin{displaymath} \widehat{\pi}(g)\widehat{v}(v) = \widehat{v}(\pi(g^{-1})v)  \end{displaymath}
for $\widehat{v} \in \widehat{V}$, $v \in V$ and $g \in G$. 

Corresponding to the representation $(\pi,V)$ is also the complex conjugate representation denoted by $(\overline{\pi},\overline{V})$. Here $\overline{V}$ is the vector space $V$ where we have defined a new scalar multiplication by $c \cdot_{\text{new}} v := \overline{c}\cdot_{\text{old}}v$ for $c \in \mathbb{C}$ and $v \in V = \overline{V}$. The representation $\overline{\pi}$ then acts on $\overline{V}$ by $\overline{\pi}(g)v = \pi(g)v$ for $v \in \overline{V}$.

\begin{dfn}\label{dfn:repcontra}
Let $G$ be a tdlc group and $(\pi,V)$ a smooth representation of $G$. We use the following terminology:
\begin{enumerate}[(i)]
   \item The \hlight{contragradient} representation of $(\pi,V)$, denoted by $(\widetilde{\pi},\widetilde{V})$, is the subrepresentation of smooth vectors in $(\widehat{\pi},\widehat{V})$; \
   \item The \hlight{Hermitian contragradient} representation of $(\pi,V)$, denoted by $(\pi^+,V^+)$, is the complex conjugate representation of $(\widetilde{\pi},\widetilde{V})$.
\end{enumerate}
\end{dfn}

The concept of contragradient representations allows us to formulate a type of matrix coefficient for general smooth representations.

\begin{dfn}\label{dfn:matrixcoef}
Let $G$ be a tdlc group and $(\pi,V)$ a smooth representation of $G$. Given $\widetilde{v} \in \widetilde{V}$ and $v \in V$, a function of the form
\begin{displaymath} \langle \pi(g) v, \widetilde{v} \rangle := \widetilde{v}(\pi(g)v)  \end{displaymath}
will be called a \hlight{matrix coefficient} of the representation $(\pi,V)$. 
\end{dfn}

\subsection{Induced representations}

Let $G$ be a tdlc group and $H$ a closed subgroup of $G$. In this subsection we discuss how to induce a smooth representation $(\sigma,W)$ of $H$ to a smooth representation of $G$. There are multiple types of induction that will be discussed. We refer the reader to \cite[$\S$1.7]{Sil79} for further information.

Fix a smooth representation $(\sigma,W)$ of $H$ for the remainder of this section. Define a space of functions associated to this representation by 
\begin{align*} V^\sigma := \{ f: G \rightarrow W : f(gh) = &\sigma(h^{-1})f(g) \: \forall h \in H, \\
&\exists U \in \COS(G) \text{ s.t. } f(ug) = f(g) \: \forall u\in U \: \forall g \in G \}.
\end{align*}
We may also consider the subspace of compactly supported functions in $V^\sigma$ which we will denote by $V^\sigma_c$. Then, the \hlight{induced representation} (resp.\ \hlight{compactly induced representation}) of $\sigma$ to $G$ is denoted by $\ind_H^G(\sigma)$ (resp.\ $\cind_H^G(\sigma)$) and acts on $V^\sigma$ (resp.\ $V^\sigma_c$) by the left regular representation of $G$. 

One notes that the representations $\ind_H^G(\sigma)$ and $\cind_H^G(\sigma)$ are isomorphic if $H$ is cocompact in $G$ i.e.\ the coset space $G/H$ is compact.

\subsection{Boundary representations and the Iwasawa decomposition}\label{sec:bdrep}

In this subsection we recall some definitions and facts about boundary representations on groups with a Gelfand pair. We also note down some results and discussion from the following paper of Nicolas Monod \cite{Mon20}. 

We summarise in the following statement the main results of the paper \cite{Mon20}.

\begin{thm}\label{thm:monod}\cite{Mon20}
Let $G$ be a locally compact group and $K \le G$ a compact subgroup such that $(G,K)$ is a Gelfand pair. Then, there exists a cocompact amenable subgroup $P \le G$ such that $G$ admits an Iwasawa decomposition $G=KP$. Furthermore, if $P$ is chosen to be maximal, then we have homeomorphisms $\partial G \cong G/P \cong K/K \cap P$, where $\partial G$ denotes the Furstenberg boundary of $G$. 
\end{thm}

Suppose that we have a Gelfand pair $(G,K)$ where $G$ is non-amenable and choose $P$ a maximal cocompact amenable subgroup of $G$ such that $G=KP$. Let $\mu$ denote the measure on $G/P$ obtained by pushing forward the normalised Haar measure on $K$ via the homeomorphism $G/P \cong K/K\cap P$. The measure $\mu$ on $G/P$ is $K$-invariant and the $G$-action preserves the measure class of $\mu$. However, $\mu$ is not $G$-invariant.

Now let $\Delta_P$ denote the modular function on $P$. This modular function is always non-trivial as a consequence of assuming that $G$ is non-amenable (see, for example, \cite[Theorem 2.51]{Fol16} for more details). Define a function $\rho$ on $G$ by $\rho(kp) := \Delta_P(p)$ for $k \in K$ and $p \in P$. Then, by the results of \cite[$\S$2.6]{Fol16}, $\rho$ is a continuous function on $G$ and we have that 

\begin{displaymath}  \frac{\dd g\mu}{\dd \mu}(xP) = \frac{\rho(g^{-1}x)}{\rho(x)}. \end{displaymath}
One should note that the function defined by 
\begin{displaymath} c: G \times G/P \rightarrow \mathbb{R}_{> 0}, (g,xP) \mapsto \frac{\dd g\mu}{\dd \mu}(xP) = \frac{\rho(g^{-1}x)}{\rho(x)}\end{displaymath}
is well known to satisfy the cocycle identity. 


We now define the following boundary representations, maintaining the notation as given above. In the following definition, given $p \in [1,\infty)$, we denote by $\B(L^p(G/P))$ the space of bounded operators on the Banach space $L^p(G/P)$.

\begin{dfn}
Given $p \in [1,\infty)$ and $z \in \mathbb{C}$, we define a representation $\pi_{z,p}: G \rightarrow \B(L^p(G/P))$ whose action on $f \in L^p(G/P)$ is given by 
\begin{displaymath} \pi_{z,p}(g)f(xP) := \bigg(\frac{\dd g\mu}{\dd \mu}(xP)\bigg)^z f(g^{-1}xP) \end{displaymath}
for $g,x \in G$. In the case when $p=2$, we will denote the representation by $\pi_z$ instead of $\pi_{z,2}$.
\end{dfn}

From now on, any representation of the form of a representation in the above definition will be referred to as a \hlight{boundary representation} on the Gelfand pair $(G,K)$. We remark that for a general complex number $z$, a boundary representation need not be unitary nor irreducible, even in the case when $p=2$.

We now show that these boundary representations are continuous.

\begin{prop}
For every $p \in [1,\infty)$ and $z \in \mathbb{C}$, the representation $(\pi_{z,p}, L^p(G/P))$ is (strongly) continuous i.e. for all $f \in L^p(G/P)$, the map 
\begin{displaymath} G \rightarrow L^p(G/P), g \mapsto \pi_{z,p}(g) f \end{displaymath} 
is continuous with respect to the group topology on $G$ and the norm topology on $L^p(G/P)$. 
\end{prop}

\begin{proof}
Throughout the proof we fix $p \in [1,\infty)$ and $z \in \mathbb{C}$.

Let $U$ be a symmetric compact neighbourhood of the identity in $G$. We first claim that the map $g \mapsto \norm{\pi_{z,p}(g)}_{\B(L^p(G/P))}$ is bounded on $U$. Indeed, since the function $\rho$ is continuous, and $U \times G/P$ is a compact subset of $G \times G/P$, the map $(g,xP) \mapsto\frac{\dd g\mu}{\dd \mu}(xP) = \frac{\rho(g^{-1}x)}{\rho(x)}$ is bounded by some constant $M >0$ on $U \times G/P$. Then, we have that, for any $f \in L^p(G/P)$ and $g \in U$
\begin{align*} \norm{\pi_{z,p}(g)f}_{L^p(G/P)} &= \bigg( \int_{G/P} \bigg\lvert  \bigg(\frac{\dd g\mu}{\dd \mu}(xP)\bigg)^z f(g^{-1}xP) \bigg\rvert^p \dd \mu(xP) \bigg)^{1/p} \\
& \le M^z \bigg( \int_{G/P} \bigg\lvert  f(g^{-1}xP) \bigg\rvert^p \dd \mu(xP) \bigg)^{1/p} \\
& = M^z \bigg( \int_{G/P} \bigg\lvert  f(xP) \bigg\rvert^p \frac{\dd g^{-1} \mu}{\dd \mu}(xP)\dd \mu(xP) \bigg)^{1/p} \\
& \le M^{\frac{pz+1}{p}} \bigg( \int_{G/P} \bigg\lvert  f(xP) \bigg\rvert^p \dd \mu(xP) \bigg)^{1/p} \\
& = M^{\frac{pz+1}{p}}  \norm{f}_{L^p(G/P)}.
\end{align*} 

Thus, by definition of the operator norm, we have that $\norm{\pi_{z,p}(g)}_{\B(L^p(G/P))} \le M^{\frac{pz+1}{p}}$ for all $g \in U$. 

Now, to prove the proposition, it suffices to show that $\pi_{z,p}$ is continuous at the identity. In particular, we need to show that, for every $f \in L^p(G/P)$, and every $\epsilon > 0$, there exists a neighbourhood $W$ of the identity in $G$ such that 
\begin{displaymath} \norm{\pi_{z,p}(g)f - f}_{L^p(G/P)} \le \epsilon  \end{displaymath}
for all $g \in W$. So fix $f \in L^p(G/P)$ and $\epsilon > 0$ for the remainder of the proof. We will now find such a neighbourhood $W$.

Since $G/P$ is compact, the continuous functions on $G/P$ are dense in $L^p(G/P)$. So we may find a continuous function $f'$ on $G/P$ such that 
\begin{displaymath} \norm{f' -f}_{L^p(G/P)} \le \frac{\epsilon}{2(M^{\frac{pz+1}{p}} +1)}. \end{displaymath}
Also, by continuity of $\rho$ and $f'$, we may find neighbourhoods $U_1$ and $U_2$ of the identity in $G$ such that 
\begin{displaymath} \bigg\lvert\bigg(\frac{\rho(g^{-1}x)}{\rho(x)}\bigg)^zf'(g^{-1}xP) -f'(g^{-1}xP) \bigg\rvert \le \frac{\epsilon}{4}. \end{displaymath}
for all $(g,xP) \in U_1 \times G/P$ and
\begin{displaymath} \lvert f'(g^{-1}xP) -f'(xP) \rvert \le \frac{\epsilon}{4} \end{displaymath}
for all $(g,xP) \in U_2 \times G/P$. Now let $\lambda_G$ denote the left regular representation of $G$. Then, for all $g \in U_1 \cap U_2$, we have that 
\begin{align*}
\norm{\pi_{z,p}(g)f' - f'}_{L^p(G/P)} &= \norm{\pi_{z,p}(g)f' -\lambda_G(g)f' + \lambda_G(g)f'- f'}_{L^p(G/P)} \\
& \le \norm{\pi_{z,p}(g)f' -\lambda_G(g)f'}_{L^p(G/P)} + \norm{\lambda_G(g)f'- f'}_{L^p(G/P)} \\
&= \bigg( \int_G \bigg\lvert \bigg(\frac{\rho(g^{-1}x)}{\rho(x)}\bigg)^z f'(g^{-1}xP) - f'(g^{-1}xP) \bigg\rvert^p \dd \mu(xP) \bigg)^{1/p} \\
&\hspace{1.5cm}+ \bigg( \int_G \bigg\lvert f'(g^{-1}xP) - f'(xP) \bigg\rvert^p \dd \mu(xP) \bigg)^{1/p} \\
& \le \bigg( \frac{\epsilon^p}{4^{p}} \bigg)^{1/p} +\bigg( \frac{\epsilon^p}{4^{p}} \bigg)^{1/p} = \frac{\epsilon}{4} + \frac{\epsilon}{4}  = \frac{\epsilon}{2}.
\end{align*}
Finally, setting $W := U \cap U_1 \cap U_2$, we have that, for all $g \in W$, 
\begin{align*} &\norm{\pi_{z,p}(g) f - f}_{L^p(G/P)} \\
&\le \norm{\pi_{z,p}(g) f - \pi_{z,p}(g)f'}_{L^p(G/P)} + \norm{\pi_{z,p}(g) f' - f'}_{L^p(G/P)} + \norm{f' - f}_{L^p(G/P)} \\
& \le (\norm{\pi_{z,p}(g)}_{\B(L^p(G/P))} + 1) \norm{f - f'}_{L^p(G/P)} + \norm{\pi_{z,p}(g) f' - f'}_{L^p(G/P)} \\
& \le (M^{\frac{pz+1}{p}}  +1) \frac{\epsilon}{2(M^{\frac{pz+1}{p}} +1)}+\frac{\epsilon}{2} = \epsilon.
\end{align*}
This completes the proof.
\end{proof}

The point of the remainder of this subsection is to show that if $G$ is a tdlc group then the subrepresentation of smooth vectors in a boundary representation is naturally isomorphic to parabolically inducing a certain character from $P$. So for the remainder of this subsection, we assume that $G$ is a tdlc group. 

Now let $z \in \mathbb{C}$ and consider the representation $\sigma_z := \ind_P^G(\Delta_P^{-z})$, where the induction is as defined in the last section. Since $P$ is cocompact in $G$ by definition, we also have that $\sigma_z \cong \cind_P^G(\Delta_P^{-z})$. 

Let $V_z$ denote the representation space of the representation $\sigma_z$. By definition of $V_z$ in the last section, we may restrict functions in $V_z$ to $K$. One checks that restricting functions in $V_z$ to $K$ defines a map
\begin{displaymath} R_K : V_z \rightarrow C^\infty(K) \end{displaymath}
where $C^\infty(K)$ denotes the space of locally constant functions on $K$. The image of this map is precisely those functions in $C^\infty(K)$ who are right translation invariant by $K \cap P$. Thus, $R_K$ may in fact be seen as a map of the form
\begin{displaymath} R_K : V_z \rightarrow C^\infty(K/K \cap P) \end{displaymath}
where we view $C^\infty(K/K \cap P)$ as a subspace of $C^\infty(K)$. Since our group $G$ admits an Iwasawa decomposition $G=KP$, one can check using this that $R_K$ is a bijection and hence identifies $V_z$ with $C^\infty(K/K \cap P)$. Consequently, the isomorphism type of the representation space $V_z$ does not depend on the complex parameter $z$.

Define a norm on $V_z$ by
\begin{displaymath} \norm{f}_{K,p} := \bigg( \int_K \lvert R_K(f)(k) \rvert^p \dd k \bigg)^{1/p} \end{displaymath}
for $f \in V_z$. The completion of $V_z$ with respect to this norm is a Banach space which we denote by $E_{z,p}$. The map $R_K$ extends by continuity to $E_{z,p}$ and gives an isomorphism of $E_{z,p}$ with $L^p(K/K \cap P)$ i.e.\ a linear isometric isomorphism
\begin{displaymath} R_K: E_{z,p} \rightarrow L^p(K/K \cap P). \end{displaymath}
Also, since $\sigma_z$ acts on $V_z$ via the left regular representation, it is continuous with respect to the norm $\norm{\:\cdot\:}_{K,p}$. Thus we may extend the representation $\sigma_z$ to $E_{z,p}$ by continuity, and the resulting representation of $G$ will be denoted by $\sigma_{z,p}$.

We define the following map which we shall use in the following proposition:

\begin{displaymath}\label{equ:bdhomeo} \varphi: K/K\cap P \rightarrow G/P, k(K\cap P) \mapsto kP. \end{displaymath}

The map $\varphi$ is a homeomorphism. Then, maintaining the notation as laid out in this subsection, we have the following. 

\begin{prop}\label{prop:bdrep1}
For every $p \in [1, \infty)$ and $z \in \mathbb{C}$, the linear map 
\begin{displaymath} U_{z,p}: E_{z,p} \rightarrow L^p(G/P), f \mapsto R_K(f)\circ \varphi^{-1} \end{displaymath}
is an isometric isomorphism of Banach spaces. Furthermore, it intertwines the representations $\pi_{z,p}$ and $\sigma_{z,p}$ i.e.\
\begin{displaymath} U_{z,p}\sigma_{z,p}(g) = \pi_{z,p}(g)U_{z,p}  \end{displaymath}
for all $g \in G$. Thus $\sigma_{z,p} \cong \pi_{z,p}$.
\end{prop}

\begin{proof}
One checks easily from the definitions that $U_{z,p}$ is an isometric isomorphism. We now show that $U_{z,p}$ intertwines the representations $\pi_{z,p}$ and $\sigma_{z,p}$. It suffices to show, by density of $V_z$ in $E_{z,p}$, that given $f \in V_z$,
\begin{displaymath} U_{z,p}(\sigma_{z,p}(g)(f)) = \pi_{z,p}(g)(U_{z,p}(f)) \end{displaymath}
for all $g \in G$.
Let $x \in G$ and write $x = kp$ for some $k \in K$ and $p \in P$. Note that $U_{z,p}(f)(xP) = f(k)$. Given $g \in G$, write $g^{-1}k = ab$ for $a \in K$ and $b \in P$. Then, $g^{-1}x = g^{-1}kp = abp$. Since $bp \in P$, we have that
\begin{displaymath} \rho(g^{-1}x) = \Delta_P(bp) = \Delta_P(b) \Delta_P(p) = \rho(g^{-1}k)\rho(x) \end{displaymath}
from which it follows that $\rho(g^{-1}k) = \frac{\rho(g^{-1}x)}{\rho(x)}$. We then compute that
\begin{align*}
U_{z,p}(\sigma_{z,p}&(g)(f))(xP) = (\sigma_{z,p}(g)(f))(k) = f(g^{-1}k) = f(ab)  \\
&= \Delta_P^{-z}(b^{-1})f(a) = \Delta_P^{z}(b)f(a) =\rho(g^{-1}k)^zf(a) = \bigg(\frac{\rho(g^{-1}x)}{\rho(x)}\bigg)^zf(a).
\end{align*}

Similarly, we compute that
\begin{align*}
(\pi_z(g)&(U_{z,p}(f)))(xP) = \bigg(\frac{\rho(g^{-1}x)}{\rho(x)}\bigg)^z U_{z,p}(f)(g^{-1}xP) \\
&= \bigg(\frac{\rho(g^{-1}x)}{\rho(x)}\bigg)^z U_{z,p}(f)(g^{-1}kP) = \bigg(\frac{\rho(g^{-1}x)}{\rho(x)}\bigg)^z f(a).
\end{align*}

Thus it follows that $U_{z,p}\sigma_{z,p}(g) = \pi_{z,p}(g)U_{z,p}$ on the dense subspace $V_z$ for all $g \in G$. By continuity, it follows that $U_{z,p}$ intertwines $\sigma_{z,p}$ and $\pi_{z,p}$. This completes the proof. 
\end{proof}

\subsection{A theorem of Godement}\label{sec:Godement}

In this section we discuss a result from \cite{God52} that we will use in the proof of Theorem A. Some of the notation and terminology in \cite{God52} is dated and requires some effort to understand, so we do this to clarify the result from Godement we use and to clarify our proof of Theorem A. 

Let $G$ be a locally compact group and $K \le G$ a compact subgroup. Given $\sigma \in \widehat{K}$, we write $p_\sigma := \dim(\sigma) \, \overline{\chi_\sigma}$, where $\chi_\sigma := {\rm Tr}(\sigma)$ is the character of $\sigma$. Note that $p_\sigma \in L^1(K)$ is a self-adjoint idempotent. 

Let $\pi$ be a topologically irreducible representation of $G$ on a Banach space $E$ and define
\begin{displaymath} P_\sigma := \pi(p_\sigma) = \int_K \pi(k)p_\sigma(k) \dd k. \end{displaymath}
The operator $P_\sigma$ is a projection of $E$ onto a closed subspace that we denote by $E(\sigma)$. The subspace $E(\sigma)$ is precisely the $\sigma$-isotypic subspace of the representation $\pi|_K$.

Godement proves the following fact in his article (see \cite[$\S$8]{God52}).

\begin{prop}\cite[$\S$8]{God52}\label{GodementThm}
Let $G$ be a locally compact group and $K$ a compact subgroup of $G$. Let $\pi$ be a topologically irreducible uniformly bounded representation of $G$ on a Banach space $E$. Suppose that for every $\sigma \in \widehat{K}$ we have $\dim(E(\sigma)) < \infty.$ Then $\ker_{L^1(G)}(\pi)$ is maximal amongst the collection of closed ideals of $L^1(G)$. 
\end{prop}

\begin{rem}
In a unital Banach algebra, maximal ideals are automatically closed, so maximal and maximal closed are the same thing. However, in the non-unital setting, it can happen that an ideal is maximal closed but not maximal in the set of all proper ideals (see \cite[pg.\ 515-516]{God52}.)
\end{rem}

We now show that if $(G,K)$ is a Gelfand pair and $\pi_{z,p}$ is a boundary representation of $G$ for some $z \in \mathbb{C}$, then the finite-dimensionality assumption of Proposition \ref{GodementThm} is satisfied.

To do this, let $G$ be a locally compact group and $p \in [1,\infty)$. The left-regular representation of $G$ on $L^p(G)$ will be denoted by $\lambda_{G,p}$. 

\begin{lem}
Let $K$ be a compact group, $\sigma \in \widehat{K}$ and $p \in [1,\infty)$. Then $\lambda_{K,p}(p_\sigma)$ is finite rank.
\end{lem}

\begin{proof}
Given $\tau \in \widehat{K}$, let $F_\tau$ be the span of the matrix coefficients of $\tau$ inside $L^p(K)$. Since $\tau$ must be finite dimensional, the subspace $F_\tau$ is finite dimensional, hence $F_\tau$ is closed. By \cite[Theorem 5.11]{Fol16},
$\spn \bigcup_{\tau \in \widehat{K}} F_\tau$ is dense in $L^p(K)$. Now consider a finite linear combination $f:= \sum_{i=1}^n f_{\tau_i}$ where $\tau_1,\dots,\tau_n \in \widehat{K}$ and $f_{\tau_i} \in F_{\tau_i}$ for each $i$. Then the orthogonality relations imply that 
$$\lambda_{K,p}(p_\sigma)f = d_\sigma \sum_{i=1}^n \overline{\chi_\sigma}*f_{\tau_i} = d_\sigma \sum_{i=1}^n \chi_{\overline{\sigma}}*f_{\tau_i} \in F_{\overline{\sigma}}.$$
Now, a generic $f \in L^p(K)$ can be written as a limit $f = \lim_\beta f_\beta$, where each $f_\beta \in  \spn \bigcup_{\sigma \in \widehat{K}} F_\sigma$. Then for each index $\beta$ we have $\lambda_{K,p}(p_\sigma)f_\beta \in F_{\overline{\sigma}}$ by the above, so that $\lambda_{K,p}(p_\sigma)f = \lim_\beta \lambda_{K,p}(p_\sigma)f_\beta \in F_{\overline{\sigma}}$ as well. Hence the image of $\lambda_{K,p}(p_\sigma)$ is contained in $F_{\overline{\sigma}}$, which proves the result.
\end{proof}


\begin{prop}\label{prop:bdadm}
Let $G$ be a locally compact group and $K$ a compact subgroup of $G$ such that $(G,K)$ is a Gelfand pair. Let $P$ be a maximal cocompact amenable subgroup of $G$ such that $G=KP$. For $z \in \mathbb{C}$ and $p \in [2,\infty)$, consider the boundary representation $\pi_{z,p}$ of $G$ on $E := L^p(G/P)$. Then, for every $\sigma \in \widehat{K}$, $\dim(E(\sigma)) < \infty$.
\end{prop}

\begin{proof}
Recall from Section \ref{sec:bdrep} that we have a homeomorphism $G/P \cong K/K\cap P$. Also, note that, for any $k \in K$ and $x \in G$, $\rho(k^{-1}x) = \rho(x)$. Thus, for $k \in K$ and $f \in L^p(G/P)$ we have 
\begin{equation}\label{eq2}
\pi_{z,p}(k)f(xP) = f(k^{-1}xP)
\end{equation}
since $\frac{\rho(k^{-1}x)}{\rho(x)}=1$.
Also, since we have a homeomorphism $G/P \cong K/K \cap P$, $L^p(G/P)$ may be identified with $L^p(K/K \cap P)$, which in turn may be identified with the subspace of $L^p(K)$ consisting of functions that are constant on cosets of $K \cap P$. Under this identification, equation \eqref{eq2} tells us that the action of $\pi_{z,p}$ when restricted to $K$ coincides with the action of the left-regular representation $\lambda_{K,p}$ on the subspace $L^p(K/K \cap P)$. Thus, given $\sigma \in \widehat{K}$, we have $\rank(\pi_{z,p}(p_\sigma)) \leq \rank(\lambda_{K,p}(p_\sigma)) < \infty$. The proposition then follows immediately.
\end{proof}


\section{Proof of Theorem A}

Throughout this section we assume the hypotheses and notation specified in Theorem A.

We start with the following lemma. A Banach $*$-algebra $\A$ is called \hlight{Wiener} if every proper closed ideal of $\A$ is contained in the kernel of a topologically irreducible non-degenerate $*$-representation of $\A$ \cite[Definition 11.5.3]{Pal01}. This just generalises the definition given for $L^1(G)$ in the introduction. 

\begin{lem}\label{lem:propM}
Suppose that $\A$ is a Wiener Banach $*$-algebra. Then every maximal closed ideal of $\A$ is $*$-closed. 
\end{lem}

\begin{proof}
 Suppose that $\A$ has the Wiener property and let $I$ be a maximal closed ideal of $\A$. According to the Wiener property, there exists a non-trivial $*$-representation $\pi$ of $\A$ with $I \subseteq \ker(\pi)$. By maximality of $I$, we must have that $\ker(\pi) = I$. Since $\pi$ is a $*$-representation, $\ker(\pi)$ is $*$-closed, and so $I$ is $*$-closed. 
\end{proof}

We now prove some lemmas about boundary representations and their matrix coefficients.

\begin{lem}\label{lem:isom}
Fix $z \in \mathbb{C}$ with $0 < \Re(z) < 1$ and define $p := 1/\Re(z)$. The representation $(\pi_{z,p}, L^p(G/P))$ is an isometric representation.
\end{lem}

\begin{proof}
Indeed, for any $f \in L^p(G/P)$, one makes the following computation:
\begin{align*} &\norm{\pi_{z,p}(g)f}_{L^p(G/P)}^p = \int_{G/P} \bigg\lvert \pi_{z,p}(g) f(xP) \bigg\rvert^p \dd\mu(xP) \\
&= \int_{G/P} \bigg\lvert \bigg(\frac{\dd g\mu}{\dd \mu}(xP) \bigg)^z  f(g^{-1}xP) \bigg\rvert^p \dd\mu(xP) \\
&= \int_{G/P} \bigg( \frac{\dd g\mu}{\dd \mu}(xP) \bigg)^{p\Re(z)} \lvert f(g^{-1}xP) \rvert^p \dd\mu(xP) \\
&= \int_{G/P} \frac{\dd g\mu}{\dd \mu}(xP) \lvert f(g^{-1}xP) \rvert^p \frac{\dd \mu}{\dd g \mu}(xP)\dd\mu(g^{-1}xP) \\
&= \int_{G/P} \lvert f(g^{-1}xP) \rvert^p \dd\mu(g^{-1}xP) = \norm{f}_{L^p(G/P)}^p. \qedhere \\
\end{align*}
\end{proof}

\begin{lem}\label{lem:bdcoef}
For any $z \in \mathbb{C}$ with $0 < \Re(z) < 1$, the matrix coefficient $\varphi_z$ is bounded.
\end{lem}

\begin{proof}
Set $p:= 1/\Re(z)$. By the previous lemma, the representation $(\pi_{z,p}, L^p(G/P))$ is an isometric representation. Now consider $\mathds{1}_{G/P}$ simultaneously as the characteristic function on $G/P$ (which is an element of $L^p(G/P)$) and as the constant 1 linear functional in $L^p(G/P)^* \cong L^q(G/P)$, where $q$ is the Hölder conjugate of $p$. Then, the corresponding matrix coefficient as defined in Definition \ref{dfn:matrixcoef} is given by
\begin{displaymath} \langle \pi_{z,p}(g)\mathds{1}_{G/P}, \mathds{1}_{G/P} \rangle  = \int_{G/P} \bigg(\frac{\dd g \mu}{\dd \mu}(xP) \bigg)^z \dd \mu(xP). \end{displaymath}
But one checks easily that this matrix coefficient is precisely equal to the matrix coefficient $\varphi_z(g) = \langle \pi_{z,2}(g) \mathds{1}_{G/P}, \mathds{1}_{G/P} \rangle_{L^2(G/P)}$.
The result then follows from the fact that matrix coefficients of an isometric Banach space representation are bounded.
\end{proof}

\begin{lem}\label{lem:irr}
Let $z \in \mathbb{C}$ with $0 < \Re(z) < 1/2$. Set $p := 1/ \Re(z)$. If the representation $(\pi_{z,2},L^2(G/P))$ is topologically irreducible, then the representation $(\pi_{z,p},L^p(G/P))$ is also topologically irreducible.
\end{lem}

\begin{proof}
Assume that $(\pi_{z,2},L^2(G/P))$ is topologically irreducible. Then, since $p >2$ and $G/P$ is compact, $L^p(G/P) \subseteq L^2(G/P)$. This implies that $L^p(G/P)^\infty \subseteq L^2(G/P)^\infty$. By Proposition \ref{prop:completionrep}(iii), since $L^2(G/P)$ is topologically irreducible, $L^2(G/P)^\infty$ is algebraically irreducible. But $L^p(G/P)^\infty$ is a non-trivial submodule of $L^2(G/P)^\infty$, hence, $L^p(G/P)^\infty = L^2(G/P)^\infty$. Thus $L^p(G/P)^\infty$ is algebraically irreducible, and by applying Proposition \ref{prop:completionrep}(iii) again, we get that $(\pi_{z,p},L^p(G/P))$ is topologically irreducible. 
\end{proof}

%

We now prove Theorem A. \

\textit{Proof of Theorem A.} Suppose the hypotheses of Theorem A. Fix $z \in \mathbb{C}$ with $0<\Re(z) <1/2$ such that the representation $(\pi_{z,2},L^2(G/P))$ is topologically irreducible and the matrix coefficient $\varphi_z(g):= \langle \pi_{z,2}(g) \mathds{1}_{G/P}, \mathds{1}_{G/P} \rangle_{L^2(G/P)}$ is not positive definite. Set $p:= 1/\Re(z) \in (2,\infty)$. Then, by Lemma \ref{lem:isom} and Lemma \ref{lem:irr}, $(\pi_{z,p},L^p(G/P))$ is an isometric representation and it is topologically irreducible. 

Since $(\pi_{z,p},L^p(G/P))$ is an isometric representation, it extends to a representation of $L^1(G)$ which we also denote by $(\pi_{z,p},L^p(G/P))$. Furthermore, by Proposition \ref{prop:bdadm}, it follows that the representation $(\pi_{z,p},L^p(G/P))$ satisfies the hypotheses of Proposition \ref{GodementThm} and thus the kernel of $(\pi_{z,p},L^p(G/P))$ as a representation of $L^1(G)$, denoted by $I := \ker_{L^1(G)}(\pi_{z,p})$, is a maximal closed ideal in $L^1(G)$. To complete the proof of Theorem A, by Lemma \ref{lem:propM}, it suffices to show that the ideal $I$ is not $*$-closed.

To do this, note that by Lemma \ref{lem:bdcoef}, the matrix coefficient $\varphi_z$ is bounded and hence defines a character $\chi_{\varphi_z}$ of $L^1(K \backslash G /K)$ whose kernel, denoted $J := \ker_{L^1(K \backslash G /K)}(\chi_{\varphi_z})$, is not $*$-closed by Corollary \ref{cor:starclosed}. To complete the proof, it suffices to show that $I \cap L^1(K \backslash G /K) = J$. Indeed, since $L^1(K \backslash G /K)$ is a $*$-subalgebra of $L^1(G)$, if $I$ was $*$-closed and $I \cap L^1(K \backslash G /K) = J$, then $J$ must be $*$-closed too, which would be a contradiction.

To see that $I \cap L^1(K \backslash G /K) = J$, let $\mathds{1}_K$ denote the characteristic function on $K$. We know that $L^1(K \backslash G /K) = \mathds{1}_K * L^1(G) * \mathds{1}_K$. Now suppose that $f \in L^1(K \backslash G /K)$. Then $f = \mathds{1}_K * f * \mathds{1}_K$ and so $\pi_{z,p}(f) = \pi_{z,p}(\mathds{1}_K) \pi_{z,p}(f) \pi_{z,p}(\mathds{1}_K)$. Note that the operator $\pi_{z,p}(\mathds{1}_K)$ is the projection onto the span of $\mathds{1}_{G/P}$ in $L^p(G/P)$, thus, one sees that $\pi_{z,p}(f) =0$ if and only if $\pi_{z,p}(f)\mathds{1}_K = 0$. 

Since $\pi_{z,p}(f) = \pi_{z,p}(\mathds{1}_K) \pi_{z,p}(f) \pi_{z,p}(\mathds{1}_K)$, there must exist $\lambda(f) \in \mathbb{C}$ such that $\pi_{z,p}(f) \mathds{1}_{G/P}= \lambda(f) \mathds{1}_{G/P}$. But using the matrix coefficient as defined in Definition \ref{dfn:matrixcoef}, and viewing $\mathds{1}_{G/P}$ simultaneuosly as an element of $L^p(G/P)$ and $L^p(G/P)^* \cong L^q(G/P)$ (where $q$ is the Hölder conjugate of $p$), we have that
\begin{align*}
\lambda(f) &= \langle \pi_{z,p}(f) \mathds{1}_{G/P}, \mathds{1}_{G/P} \rangle = \int_G f(x) \langle \pi_{z,p}(x) \mathds{1}_{G/P}, \mathds{1}_{G/P} \rangle \dd x \\
&= \int_G f(x) \langle \pi_{z,2}(x) \mathds{1}_{G/P}, \mathds{1}_{G/P} \rangle_{L^2(G/P)} \dd x = \chi_{\varphi_z}(f).
\end{align*}
Thus $\pi_{z,p}(f) = 0$ if and only if $\chi_{\varphi_z}(f) = 0$. This implies that $I \cap L^1(K \backslash G /K) = J$ which completes the proof. \qed


\section{Proof of Theorem B}

In this section we complete the proof of Theorem B from the introduction.

\subsection{Proof of Theorem B(i)} 

We now give a proof of Theorem B(i). To start, we introduce some preliminary notation and results from the literature that will be used in the proof. 

Let $d_1$ and $d_2$ be two natural numbers $\ge 2$ and denote by $T_{d_1,d_2}$ the semi-regular tree of degree $(d_1,d_2)$. In particular, $T_{d_1,d_2}$ is the infinite bipartite tree such that, with respect to the associated bipartition of the vertex set $VT_{d_1,d_2} = X \sqcup Y$, all the vertices in $X$ have degree $d_1$, and all the vertices in $Y$ have degree $d_2$. We refer the reader to \cite[Chapter 1]{FTN91} or \cite[Chapitre I]{Cho94} for further information on (semi-)regular trees and groups acting on them. Standard terminology and results from these references will be used in this section.

The following is a consequence of the results in \cite{Nev91}.

\begin{lem}\cite{Nev91}
Let $X$ be a connected locally finite graph with infinitely many ends and suppose that $G \le \Aut(X)$ is a closed non-compact subgroup acting transitively on the boundary of $X$. Then, there exists natural numbers $d_1 \ge 2$ and $d_2 \ge 3$, such that, $G$ admits a quotient onto a closed subgroup $H \le \Aut(T_{d_1,d_2})$ which acts transitively on the boundary of $T_{d_1,d_2}$ and has at most two orbits on the vertices.
\end{lem}

It is well known that the property of a group being Wiener is preserved under taking quotients \cite[Theorem 11.5.4]{Pal01}, thus, to complete the proof of Theorem B(i), it suffices to show that any group satisfying the properties of $H$ in the lemma is not Wiener. We now proceed with proving this fact. 

So for the remainder of this section we fix natural numbers $d_1 \ge 2$ and $d_2 \ge 3$, and fix a closed subgroup $G \le \Aut(T_{d_1,d_2})$ which is not compact, acts transitively on the boundary of $T_{d_1,d_2}$, and has at most two orbits on the vertices. Now let $v$ be a vertex of $T_{d_1,d_2}$ and $\infty$ an end of $T_{d_1,d_2}$. We set $K$ and $P$ to be the subgroups of automorphisms of $G$ that fix $v$ and $\infty$ respectively. It is well known that $(G,K)$ is a Gelfand pair \cite[Lemma 3.2.12]{Sem23-thesis} and that $P$ is a maximal cocompact amenable subgroup of $G$ such that $G=KP$.

Let $(v_n)_{n \in \mathbb{N}}$ be a ray in the tree $T_{d_1,d_2}$ which lies in the end $\infty$. For every $p \in P$, and $n$ sufficiently large, there exists a $k \in \mathbb{Z}$ such that $p(v_n) = p(v_{n+k})$. We then define a function $\nu: P \rightarrow \mathbb{Q}$ given by $\nu(p) := k/2$. This can be shown to be a homomorphism and its definition does not depend on the choice of the ray $(v_n)_{n \in \mathbb{N}}$ (see \cite[Chapitre I]{Cho94}). Then, following the work of \cite[Chapter 2]{Cho94}, one defines for every $\lambda \in \mathbb{C}$ a character $\chi^\lambda$ on $P$ given by
\begin{displaymath} \chi^\lambda(p) := (\sqrt{d_1d_2}\lambda)^{\nu(p)}. \end{displaymath}
We then consider, as done again in \cite[Chapter 2]{Cho94}, the smoothly induced representation $\pi^\lambda := \ind_P^G(\chi^\lambda)$ whose representation space is given by
\begin{displaymath} V^\lambda = \{ f \in C^\infty(G) : f(gp) = \chi^\lambda(p)f(g), g \in G, p \in P \}. \end{displaymath}
As discussed in the preliminaries, restricting functions in $V^\lambda$ to $K$ gives rise to an isomorphism $R_K: V^\lambda \rightarrow C^\infty(K/K\cap P)$, where we consider $C^\infty(K/K\cap P)$ as a subspace of $C^\infty(K)$. On $V^\lambda$, one then defines a Hermitian form by
\begin{displaymath} \langle f,g\rangle_K := \int_K f(k)\overline{g(k)}\dd k. \end{displaymath}
The completion of $V^\lambda$ with respect to this form is a Hilbert space denoted by $\H^\lambda$, and it is shown in \cite[Proposition 2.1.2]{Cho94} that $\pi^\lambda$ extends to a bounded representation on $\H^\lambda$, denoted by $\Pi^\lambda$. 

Now let $\varphi: K/K \cap P \rightarrow G/P$ denote the canonical homeomorphism. It is a consequence of the results in \cite[Section 2.2]{Cho94} that the map $R_K\circ \varphi^{-1}: V^\lambda \rightarrow C^\infty(G/P)$ extends to an isometric isomorphism $U: \H^\lambda \rightarrow L^2(G/P)$. Furthermore, it is a consequence of \cite[Proposition 2.2.1]{Cho94}, that if $z \in \mathbb{C}$ such that $\lambda = (d_1d_2)^{z-\frac12}$, then $U\Pi^\lambda(g) = \pi_z(g)U$ for all $g \in G$, where $(\pi_z,L^2(G/P))$ is the boundary representation as defined in the preliminaries of this article. In particular, $\Pi^\lambda$ and $\pi_z$ are equivalent representations for $\lambda = (d_1d_2)^{z-\frac12}$. 

Using \cite[Théorème 2.4.6]{Cho94}, one deduces that $\pi_z$ is topologically irreducible provided the following hold:
\begin{itemize}
   \item $z \ne 1 + \frac{2\pi i k}{\log(d_1d_2)}$ for $k \in \mathbb{Z}$;
   \item $z \ne \frac{2\pi i k}{\log(d_1d_2)}$ for $k \in \mathbb{Z}$;
   \item $z \ne \frac{\log(d_1)}{\log(d_1d_2)} + \frac{\pi i (2k+1)}{\log(d_1d_2)}$ for $k \in \mathbb{Z}$;
   \item $z \ne \frac{\log(d_2)}{\log(d_1d_2)} + \frac{\pi i (2k-1)}{\log(d_1d_2)}$ for $k \in \mathbb{Z}$.
\end{itemize}

Now, let $\mathds{1}_K$ denote the characteristic function on $K$ which is contained in the space $C^\infty(K/K\cap P) \cong V^\lambda \subset \H^\lambda$. Then, one checks by definition of $U$ that for all $g \in G$
\begin{align*}
\psi_\lambda(g) := \langle \Pi^\lambda(g) \mathds{1}_K, &\mathds{1}_K\rangle = \langle U^{-1}\pi_{z}(g)U\mathds{1}_K, \mathds{1}_K \rangle \\
&= \langle \pi_{z}(g)U\mathds{1}_K, U\mathds{1}_K \rangle = \langle \pi_{z}(g)\mathds{1}_{G/P}, \mathds{1}_{G/P} \rangle = \varphi_z(g)
\end{align*}
where we again have that $\lambda = (d_1d_2)^{z-\frac12}$.

By \cite[Proposition 3.2.2(1)]{Cho94}, if $\lambda_1, \lambda_2 \in \mathbb{C}$, $\psi_{\lambda_1} = \psi_{\lambda_2}$ if and only if $\lambda_1 = \lambda_2$ or $\lambda_1 = \lambda_2^{-1}$. This implies that, for $z_1,z_2 \in \mathbb{C}$, $\varphi_{z_1} = \varphi_{z_2}$ if and only if $z_1 = z_2 + \frac{2\pi i k}{\log(d_1d_2)}$ or $z_1 = z_2 +1 + \frac{2\pi i k}{\log(d_1d_2)}$ for some $k \in \mathbb{Z}$. Then, one checks that for $z \in \mathbb{C}$, $\varphi^*_z(g) = \overline{\varphi_z(g^{-1})} = \varphi_{-\overline{z}}(g)$ for all $g \in G$. It is then easy to find $z \in \mathbb{C}$ with $0 < \Re(z) < 1/2$, $\Re(z) \ne \frac{\log(d_1)}{\log(d_1d_2)}$, $\Re(z) \ne \frac{\log(d_2)}{\log(d_1d_2)}$, such that $\varphi_z \ne \varphi_{-\overline{z}}$. This then implies by Proposition \ref{prop:posdef}, for this choice of $z$, that $\varphi_z$ is not positive definite. Furthermore, since $0 < \Re(z) < 1/2$, $\Re(z) \ne \frac{\log(d_1)}{\log(d_1d_2)}$ and $\Re(z) \ne \frac{\log(d_2)}{\log(d_1d_2)}$, this ensures that $\pi_z$ is topologically irreducible. Thus, by Theorem A, $G$ is not Wiener. 

\subsection{Proof of Theorem B(ii)}

Throughout this section we assume some rudimentary knowledge of the theory of reductive groups and their representation theory. We use \cite{Car79} as the main reference. To start, we set the following notation which will be used throughout the proof:
\begin{itemize}
   \item $G$ will denote a split reductive algebraic group over a non-archimedean local field $k$ with residue degree $q$;
   \item $A$ is a maximal split torus in $G$, $M$ the centraliser of $A$ in $G$, $N(A)$ the normaliser of $A$ in $G$, and $W := N(A)/M$ the associated Weyl group;
   \item $\mathcal{B}$ is the Bruhat-Tits building associated to $G$, $\mathcal{A}$ the fundamental apartment in $\mathcal{B}$ associated to $A$, $x_0$ a special vertex in $\mathcal{A}$, and $K$ the maximal compact open subgroup of elements in $G$ which fix $x_0$;
   \item $P$ is a minimal parabolic subgroup in $G$ such that $G=KP$ and $N$ the unipotent radical of $P$ so that $P=MN$.
\end{itemize}

To start the proof, first note that it is well known that $(G,K)$ is a Gelfand pair. See for example \cite[Corollary 4.1]{Car79}.  

Now the group $P$ is non-unimodular and we let $\Delta_P$ be the modular function on $P$. To show that $G$ is not a Wiener group, it suffices to show, by combining Theorem A, Proposition \ref{prop:completionrep} and Proposition \ref{prop:unitarizable}, that the smooth induced representation $\sigma_z := \ind_P^G(\Delta_P^{-z})$ is algebraically irreducible and not unitarizable for some $z \in \mathbb{C}$ with $0 < \Re(z) < 1/2$. 

Let $w_\ell$ be the longest word of the Weyl group $W$ and consider the function
\begin{equation}\label{eqn:c-func} \mathbb{C} \rightarrow \mathbb{C}, z \mapsto c_{w_\ell}(w_\ell\cdot\Delta_P^z) c_{w_\ell}(\Delta_P^z) \end{equation}
where the function $c_{w_\ell}$ is as defined in \cite[$\S$3]{Cas80}. One checks by definition of the function $c_{w_\ell}$ that the map $z \mapsto c_{w_\ell}(w_\ell\cdot\Delta_P^z) c_{w_\ell}(\Delta_P^z)$ is meromorphic and hence can have at most countably many zeroes. It then follows from \cite[Proposition 3.5(b)]{Cas80} that the representation $\sigma_z$ is algebraically irreducible everywhere except at countably many points. 

Thus, to find a $z \in \mathbb{C}$ with $0 < \Re(z) < 1/2$ such that $\sigma_z$ is algebraically irreducible and not unitarizable, it suffices to show that there are uncountably many $z \in \mathbb{C}$ with $0 < \Re(z) < 1/2$ such that $\sigma_z$ is not unitarizable. To do this, it further suffices to show that there are uncountably many $z \in \mathbb{C}$ with $0 < \Re(z) < 1/2$ such that $\sigma_z$ is not Hermitian i.e.\ $\sigma_z^+ \ncong \sigma_z$.

Note that $\sigma_{-z}^+ = \ind_P^G(\Delta_P^z)^+ = \ind_P^G(\Delta_P^{-\overline{z}}) = \sigma_{\overline{z}}$. So we need to show that there are uncountably many $z \in \mathbb{C}$ with $0 < \Re(z) < 1/2$ such that $\ind_P^G(\Delta_P^z) \ncong \ind_P^G(\Delta_P^{-\overline{z}})$. Since the representation $\sigma_z$ is spherical whenever it is irreducible, it follows from Theorem 4.2(b) and Theorem 4.3 in \cite{Car79} that the representations $\ind_P^G(\Delta_P^z)$ and $\ind_P^G(\Delta_P^{-\overline{z}})$ are equivalent if and only if there exists $w \in W$ such that $w \cdot \Delta_P^{z} = \Delta_P^{-\overline{z}}$ (provided that $\sigma_z = \ind_P^G(\Delta_P^z)$ is irreducible).

We claim that, if there exists $w \in W$ such that $w \cdot \Delta_P^{z} = \Delta_P^{-\overline{z}}$, then $z$ must be either purely imaginary or purely real. This will then imply Theorem B(ii), since it implies that there are uncountably many $z \in \mathbb{C}$ with $0 < \Re(z) < 1/2$ such that $\sigma_z$ is not unitarizable and is algebraically irreducible. So to prove the claim, it is well known that $P=MN$, and since $G$ is split, the modular function on $P$ is given by the expression 
\begin{displaymath} \Delta_P(mn) = q^{2\rho(m\cdot x_0)}  \end{displaymath}
where $\rho$ is the half sum of the positive roots associated to $G$ and $q$ is the residue characteristic of $k$ (see Section 3.5 and Equation $24_c$ in \cite{Car79} for further information). Then, if $w \cdot \Delta_P^{z} = \Delta_P^{-\overline{z}}$, this implies that $2z (w\cdot \rho) = -2\overline{z}\rho$, which in turn implies that $w\cdot \rho = \frac{-\overline{z}}{z}\rho$. So $w\cdot \rho$ must be a complex multiple of $\rho$, and in particular, this implies that either $w\cdot\rho = \rho$ or $w\cdot \rho = -\rho$. Thus, either we have $\frac{-\overline{z}}{z}=1$ or we have $\frac{-\overline{z}}{z}=-1$, and in particular, $z$ is either purely imaginary or purely real. This completes the proof.


\section{Acknowledgements}

We would like to thank Anne-Marie Aubert for her help sourcing references regarding the representation theory of the reductive groups in the proof of Theorem B. We also thank Pierre-Emmanuel Caprace, Dougal Davis, Simon Marshall and Maarten Solleveld for useful discussions during the course of this work. 

\bibliographystyle{amsplain}
\bibliography{non-wiener}


\end{document}